\documentclass[reqno]{amsart}
\usepackage[scale=0.75, centering, headheight=14pt]{geometry}
\usepackage[latin1]{inputenc}
\usepackage[T1]{fontenc}
\usepackage{lmodern}
\usepackage[english]{babel}
\usepackage{esint}
\usepackage{mathtools}

\usepackage{xcolor,amsmath,mathrsfs,amssymb,accents,graphicx,epstopdf, amsthm}
\usepackage{ bbold }
\usepackage{comment}
\usepackage{float}
\usepackage{subfig}
\usepackage{tikz}

\usepackage{tikz}
\usepackage{bbm}
\usepackage{amsmath,amssymb,amsfonts,amsthm}
\usepackage{mathtools,accents}
\usepackage{mathrsfs}
\usepackage{xfrac}
\usepackage{array} 
\usepackage{aliascnt}
\usepackage{booktabs} 
\usepackage{array} 

\usepackage{verbatim} 
\usepackage{subfig} 

\usepackage{mathrsfs, dsfont}
\usepackage{amssymb}
\usepackage{amsthm}
\usepackage{amsmath,amsfonts,amssymb,esint}
\usepackage{graphics,color}
\usepackage{enumerate}
\usepackage{mathtools,centernot}

\usepackage{microtype}
\usepackage{paralist} 
\usepackage{cases}
\usepackage[initials]{amsrefs}
\allowdisplaybreaks

\usepackage{braket}
\usepackage{bm}

\usepackage[citecolor=blue,colorlinks]{hyperref}
\addto\extrasenglish{}

\usepackage{enumerate}
\usepackage{xcolor}
\usepackage{aliascnt}

\numberwithin{equation}{section}

\newtheorem{theorem}{Theorem}[section]
\newtheorem{lemma}[theorem]{Lemma}
\newtheorem{definition}[theorem]{Definition}
\newtheorem{remark}[theorem]{Remark}
\newtheorem{proposition}[theorem]{Proposition}

\newcommand{\N}{\mathbb{N}}
\newcommand{\R}{\mathbb{R}}

\newcommand{\Z}{\mathbb{Z}}

\newcommand{\Leb}[1]{{\mathscr L}^{#1}} 

\newcommand{\x}{\times}

\renewcommand{\a}{\alpha}

\newcommand{\e}{\varepsilon}

\renewcommand\div{\operatorname{div}}

\newcommand{\supp}{\operatorname{supp}}

\newcommand{\vo}{\vec{o}\@ifnextchar{^}{\,}{}}

 \newcommand{\bb}{{\mbox{\boldmath$b$}}}

 \newcommand{\uu}{{\mbox{\boldmath$u$}}}
 
 \newcommand{\ww}{{\mbox{\boldmath$w$}}}

 \newcommand{\tauV}{{\kern-3pt\tau}}

 \newcommand{\XX}{{\mbox{\boldmath$X$}}}

 \newcommand{\oVVVk}{\overline{\mbox{\boldmath$V$}}\kern-3pt}
 \newcommand{\tVVVk}{\tilde{\mbox{\boldmath$V$}}\kern-3pt}

	\title{A remark on selection of solutions for the transport equation} 
\author [J. Pitcho]{Jules Pitcho}
\address{Jules Pitcho
	\hfill\break  ENS  de Lyon, UMPA, 46 all\'ee d'Italie, 
	69364 Lyon,
	France}
\email{jules.pitcho@ens-lyon.fr}

\begin{document}
\begin{abstract}
	We prove that for bounded, divergence-free vector fields in $L^1_{loc}((0,+\infty);BV_{loc}(\R^d;\R^d))$ regularisation by convolution of the vector field selects a single solution of the transport equation for any locally integrable initial datum. We recall the vector field constructed by Depauw in \cite{Depauw}, which lies in the above class of vector fields. We show that the transport equation along this vector field has at least two bounded weak solutions for any bounded initial datum. 
\end{abstract}

	\maketitle 
	
	\section{Introduction} 
\subsection{The Cauchy problem} 
In this note, we study the initial value problem for the continuity equation posed on $[0,+\infty)\times\R^d$,
\begin{equation}\label{eqPDE}\tag{IVP} 
\left\{
\begin{split}
\partial_t \rho+ \div_x(\bb \rho)&=0,\\
\rho(0,x)&=\bar\rho(x),
\end{split}
\right.
\end{equation}
where $\bb=\bb(t,x)$ is a given vector field, $\rho=\rho(t,x)$ is an unknown real-valued function, and $\div_x$ is the divergence operator on vector fields on $\R^d$. We are interested in weak solutions of \eqref{eqPDE}.

\begin{definition}\label{defPDE}
	Consider a bounded vector field $\bb:[0,+\infty)\times\R^d\to\R^d$ and an initial datum $\bar\rho\in L^1_{loc}(\R^d)$. 
	We shall say $\rho\in L_{loc}^\infty([0,+\infty);L^1_{loc}(\R^d))$ is a weak solution to \eqref{eqPDE} along $\bb$, if for every $\phi \in C_c^\infty([0,+\infty)\times\mathbb{R}^d)$
	\begin{equation*}
	\int_0^{+\infty}\int_{\mathbb{R}^d} \rho\left(\frac{\partial \phi}{\partial t} + {\bb}\cdot\nabla_x \phi\right) \; dxdt = - \int_{\mathbb{R}^d} \bar\rho(x) \phi(0,x) \; dx.
	\end{equation*}
	If, additionally $\rho\in L_{loc}^\infty([0,+\infty);L^\infty(\R^d))$, we shall say that it is a bounded weak solution to \eqref{eqPDE} along $\bb$. 
\end{definition}
We are interested in selection of weak solutions of \eqref{eqPDE} when they are non-unique. Let us recall the classical theory of existence and uniqueness of weak solutions of \eqref{eqPDE}. 
\bigskip 
\subsection{The classical theory}
Given a bounded vector field $\bb:[0,+\infty)\times\R^d\to\R^d$, it is convenient to work with its extension by zero to $\R\times\R^d$, which we denote by $\tilde\bb$ and define as 
\begin{equation}\label{eqn_extension_zero} 
\tilde\bb(t,x):=
\left\{
\begin{split}
\bb(t,x) \qquad &{\rm if }\quad (t,x)\in[0,+\infty)\times\R^d,\\
0\qquad &{\rm if} \quad (t,x)\notin [0,+\infty)\times\R^d.
\end{split} 
\right. 
\end{equation}

When $\bb$ is locally Lipschitz continuous in $x$ with Lipschitz constants on compact sets, which are time integrable, the Cauchy-Lipschitz theorem provides unique global solutions on $\R$ to the ODE understood in the sense of distributions
\begin{equation}\label{ODE}\tag{ODE}
\left\{
\begin{split}
\partial_t\XX(t,s,x)&=\tilde \bb(t,\XX(t,s,x)),\\
\XX(s,s,x)&=x,
\end{split}
\right. 
\end{equation}
for every $s\in\R$.
These solutions are then bundled into a 2-parameter family of maps $\XX:\R\times \R\times \R^d\to \R^d$, which we will call the flow along $\bb$, and satisfies the classical stability estimate
\begin{equation*}
|x_1-x_2|\exp\Big(-\Big|\int_s^t\|\nabla_x \bb(u)\|_{L_x^\infty}du\Big|\Big)\leq 	|\XX(t,s,x_1)-\XX(t,s,x_2)|\leq |x_1-x_2|\exp\Big(\Big|\int_s^t\|\nabla_x \bb(u)\|_{L_x^\infty}du\Big|\Big),
\end{equation*}
and the group property for every $r,s,t\in\R$
\begin{equation*}
\XX(t,s,\XX(s,r,\cdot))=\XX(t,r,\cdot).
\end{equation*}

In this setting, weak solutions $\rho$ of \eqref{eqPDE} along $\bb$ are uniquely given by the classical formula
\begin{equation}\label{eqn_representation}
\rho(t,\cdot)\Leb{d}=\XX(t,0,\cdot)_\#\bar\rho\Leb{d}. 
\end{equation}

If the vector field $\bb$ satisfies suitable growth assumptions, then the global existence and uniqueness of the flow $\XX$ is still valid. For simplicity, we choose to restrict ourselves to bounded vector fields in this work. Bounded vector fields to which the classical theory applies will be refered to as smooth vector fields. On the contrary, bounded vector fields to which this theory does not apply will be called non-smooth vector fields. 

\bigskip 

\subsection{Basic definitions} 
Our purpose is to study the selection of weak solutions of \eqref{eqPDE} defined by the classical theory under regularisation of a non-smooth $\bb$. A regularisation of $\bb$ is a sequence $(\bb^k)_{k\in\N}$ in $C^\infty([0,+\infty)\times\R^d;\R^d)$ such that $\bb^k\to\bb$ in $L^1_{loc}$. This leads to the following definition. 

\begin{definition}\label{defn_well_posedness_reg}
	Consider a bounded vector field $\bb:[0,+\infty)\times\R^d\to \R^d$ and an initial datum $\bar\rho \in L^1_{loc}(\R^d)$. We shall say that the Cauchy problem \eqref{eqPDE} is well-posed along a regularisation $(\bb^k)_{k\in\N}$, if  the sequence of unique weak solutions $\rho^k$ of \eqref{eqPDE} along $\bb^k$ converges uniquely in $\mathcal{D}'((0,+\infty)\times\R^d)$ as $k\to+\infty$. 
\end{definition}
\begin{remark}
	We can similarly define well-posedness along a regularisation $(\bb^h)_{h\in I}$, where $I$ is not necessarily countable. For simplicity, in this paper we restrict our attention to regularisations indexed by the natural numbers. 
\end{remark}
For a smooth $\bb$, the Cauchy problem is well-posed along any regularisation of $\bb$ by using \eqref{eqn_representation} and the classical stability of the flow $\XX$ under smooth perturbation of $\bb$. 
If we approximate a non-smooth $\bb$ in a stronger topology than $L^1_{loc}$, then we can have stronger convergence than in $\mathcal{D}'((0,+\infty)\times\R^d)$ for the solutions along the approximation. This will be a key fact in the proof our main theorem. 

For a non-smooth $\bb$, weak limits of weak solutions of \eqref{eqPDE} along a regularisation $(\bb^k)_{k\in\N}$ are expected to be non-unique, and to depend on the choice of the regularisation. A stronger notion of well-posedness is therefore when the weak limit coincides for several regularisations of $\bb$. 
This leads to the following definition. 
\begin{definition}
	Consider a bounded vector field $\bb:[0,+\infty)\times\R^d\to \R^d$ and an initial datum $\bar\rho\in L^1_{loc}(\R^d)$. Consider a family $\mathcal{R}$ consisting of regularisations of $\bb$. We shall say that the Cauchy problem \eqref{eqPDE} is well-posed along $\mathcal{R}$, if there exists a unique { $\rho\in L_{loc}^1([0,+\infty)\times\R^d)$} such that for every $(\bb^k)_{k\in\N}\in \mathcal{R}$, the unique weak solutions $\rho^k$ of \eqref{eqPDE} along $\bb^k$ converge to $\rho$ in $\mathcal{D}'((0,+\infty)\times\R^d)$ as $k\to+\infty$.
\end{definition}

In the study of non-linear wave equations with rough initial data, a similar notion of well-posedness is used (see the reviews \cite{Tzv19Randomdatawave,Tzv23nonlinearnoise}). In fact, if the initial datum is random and rough, the Cauchy problem can be well-posed with respect to some regularisation but not with respect to another \cite{Tzv19Randomdatawave,SunTzv20CRAS}.
\bigskip

\subsection{Review of relevant results}
Let us review the known existence and uniqueness results on weak solutions of \eqref{eqPDE} when $\bb$ is non-smooth. 
The work of Ambrosio \cite{A04} following on the work DiPerna-Lions \cite{DPL89} reads as follows in the context of divergence-free vector fields.
\begin{theorem}\label{thm_uniqueness_bv} 
	Consider a bounded, divergence-free vector field $\bb:[0,+\infty)\times\R^d\to\R^d$, and an initial datum $\bar\rho\in L^\infty(\R^d)$. Assume that $\bb\in L_{loc}^1([0,+\infty);BV_{loc}(\R^d;\R^d))$. Then, there exists a unique bounded weak solution of \eqref{eqPDE}.
\end{theorem}
The uniqueness part of the above theorem was proved in \cite{A04} assuming only that for almost every $t\in\R$ the divergence at time $t$ of $\bb$ is absolutely continuous with respect to $\Leb{d}$. The proof is based on a commutator estimate showing that any bounded weak solution of \eqref{eqPDE} must be renormalised.
By a fine measure-theoretic analysis of Lagrangian representations of solutions of \eqref{eqPDE}, the hypothesis on the divergence has been further relaxed to \emph{near incompressibility} in the work of Bianchini and Bonicatto \cite{BianchiniBonicatto20}.

In general, when uniqueness of bounded weak solutions holds for \eqref{eqPDE}, then well-posedness  in the sense of Definition \ref{defn_well_posedness_reg} along any regularisation of $\bb$ is true, up to some mild assumptions on the regularisations ensuring that the solutions of \eqref{eqPDE} along the regularised vector fields remain uniformly bounded. More precisely, the following statement may be deduced from \cite{A04} (see also the Appendix for a proof).

\begin{proposition}\label{prop_well_posed}
	{ Consider an initial datum $\bar\rho\in L^\infty(\R^d)$.} Assume the hypothesis of Theorem \ref{thm_uniqueness_bv} on $\bb$. Consider a regularisation $(\bb^k)_{k\in \N}$ of $\bb$ such that 
	\begin{equation}\label{eq_lebesgue_measure}
	\XX^k(t,0,\cdot)_\#\Leb{d}\leq C\Leb{d},
	\end{equation} 
	 where the constant $C$ is uniform in $k$ and on compact time intervals. 
	Then, the Cauchy problem \eqref{eqPDE} is well-posed along the regularisation $(\bb^k)_{k\in\N}$. 
\end{proposition}
It would be interesting to establish how much \eqref{eq_lebesgue_measure} can be relaxed, and whether solutions with similar properties as those obtained by convex integration (see \cite{ML18, MS20, MoSzRenormalized, BDLC20, pitcho2021nonuniqueness,sattig2023baire,buck2022failure}) may be constructed as limit of smooth approximations of \eqref{eqPDE} when \eqref{eq_lebesgue_measure} fails by too much. These solutions are however ``pathological'' in that they are not selected by the vanishing diffusivity scheme, as shown by Bonicatto, Ciampa and Crippa in \cite{bonicatto21adv}. 

Let us now come to the non-uniqueness results. 
First Depauw constructed a bounded, divergence-free vector field $\bb_{DP}:[0,+\infty)\times\R^2\to\R^2$ in \cite{Depauw}, which is not in $L^1_{loc}([0,+\infty);BV_{loc}(\R^2;\R^2))$, and for which uniqueness of bounded weak solutions to \eqref{eqPDE} fails. { Ciampa, Crippa and Spirito then constructed in \cite{CiampaCrippaSpirito20} an unbounded, divergence-free and autonomous vector field, and a family of initial data, for which the Cauchy problem \eqref{eqPDE} is well-posed along two different regularisations of $\bb$, but the two corresponding solutions do not coincide.}
By adapting the construction of Depauw, De Lellis and Giri constructed in \cite{DeLellis_Giri22} a bounded vector field and an initial datum, for which the Cauchy problem \eqref{eqPDE} is well-posed along two different regularisations of $\bb$ (in the sense of Definition \ref{defn_well_posedness_reg}), but the two corresponding solutions do not coincide. 

This was then extended by Colombo, Crippa and Sorella in \cite{colombo2022anomalous}, where for every $\alpha\in[0,1)$, they construct divergence-free vector fields $\bb$ in $C^\alpha([0,2]\times\R^2)$, an initial datum $\bar\rho$ for which the Cauchy problem \eqref{eqPDE} is not well-posed (in the sense of Definition \ref{defn_well_posedness_reg}) along a certain regularisation of $\bb$, which is given by convolution with a standard mollifier. Although not directly related to the present work, we note that for those same vector fields, they show that the vanishing diffusivity regularisation scheme fails to select a single solution. In a recent contribution, Huysmans and Titi \cite{huysmans2023nonuniqueness} have moreover constructed a bounded vector field for which the vanishing diffusivity scheme selects a solution, which is not entropy-admissible in the sense of Dafermos \cite{Dafermos10ConsLaw}.

\bigskip
\subsection{Statement of the theorems}
For vector fields for which uniqueness of bounded weak solutions of \eqref{eqPDE} may fail, we are interested in well-posedness of the Cauchy problem along a whole regularisation class. 
We consider the regularisation class obtained by mollification of the vector field with an arbitrary standard mollifier. 
We recall that a function $\theta\in C_c^\infty(\R\times\R^d)$ is called a standard mollifier if $\theta\geq 0$, and $$\int_\R\int_{\R^d} \theta(t,x)dxdt=1.$$
We write for every $k\in\N$ 
$$\theta^k(t,x):=k^{d+1}\theta(kt,kx).$$
By a slight abuse of notation, the convolution $\bb\star \theta^k$ then denotes the restriction to $[0,+\infty)\times\R^d$ of the convolution $\tilde\bb\star \theta^k$, where we recall that $\tilde{\bb}$ defined in \eqref{eqn_extension_zero} is the extension by zero to negative times, { and we recall that $$(\tilde\bb\star \theta^k)(t,x):=\int_{\R}\int_{\R^d}\tilde\bb(t-s,x-y)\theta^k(s,y)dyds.$$}
We define the convolution regularisation class $$\mathcal{R}_{conv}:=\Big\{(\bb\star \theta^k)_{k\in\N} \,; \,\theta \in 
C_c^\infty(\R\times \R^d), \;\theta\geq 0,\; 
\int_{\R^{d+1}}
\theta(t,x)dtdx=1\Big\}.$$
\begin{remark}
	For a divergence-free vector field $\bb$ and a constant initial datum, the Cauchy problem \eqref{eqPDE} is automatically well-posed along $\mathcal{R}_{conv}$. { Indeed, in this case, for every $(\bb^k)_{k\in\N}\in\mathcal{R}_{conv}$, the constant function is the unique bounded weak solution of \eqref{eqPDE} along $\bb^k$ for every $k\in\N$. Clearly, this sequence of constant functions then converges uniquely to the same constant function.}
\end{remark}

We have the following well-posedness theorem for this regularisation class. 
\begin{theorem}\label{thm_uniqueness}
	Consider a bounded, divergence-free vector field $\bb:[0,+\infty)\times\R^d\to \R^d$ and {an initial datum $\bar\rho\in L_{loc}^1(\R^d)$.} Assume that $\bb\in L_{loc}^1((0,+\infty);{\rm BV_{loc}}(\R^d;\R^d))$. 
	Then, the Cauchy problem \eqref{eqPDE} is well-posed along the regularisation class $\mathcal{R}_{conv}$. 
\end{theorem}
{ \begin{remark} Note that well-posedness of the Cauchy problem \eqref{eqPDE} along the regularisation class $\mathcal{R}_{conv}$ does not in general imply existence of a weak solution of \eqref{eqPDE}. When $1<p\leq +\infty$, and when the initial datum is in $L^p_{loc}(\R^d)$, weak limit points of weak solutions of \eqref{eqPDE} along a regularisation of $(\bb^k)_{k\in\N}$ by convolution are necessarly weak solutions of \eqref{eqPDE}. However, when the initial datum is in $L^1_{loc}(\R^d)$, we cannot directly pass into the limit in the weak formulation of \eqref{eqPDE} because $\bb^k$ need not converge strongly to $\bb$ in $L^\infty_{loc}((0,+\infty)\times\R^d)$.
	\end{remark}
}
{\begin{remark} We stress that our result holds for a locally integrable initial datum, whereas in Theorem~\ref{thm_ambrosio} due to Ambrosio, the initial datum needs to be essentially bounded. This is because the notion of well-posedness along the regularisation class $\mathcal{R}_{conv}$ is weaker than the notion of existence and uniqueness of weak solutions in the functional class $L^\infty_{loc}([0,+\infty);L^1_{loc}(\R^d))$. It is in fact a strictly weaker notion as the weak solutions constructed by convex integration \cite{ML18, MS20, MoSzRenormalized, BDLC20, pitcho2021nonuniqueness} cannot be constructed by regularisation of the vector field by convolution.
	\end{remark}
}
It is interesting to note that for non-linear wave equations with random and rough initial data \cite{Tzv19Randomdatawave,Tzv23nonlinearnoise}, the class of regularisations by convolution also plays an important role, although for different reasons than in this work.

The bounded, divergence-free vector field $\bb_{DP}:[0,+\infty)\times\R^2\to\R^2$ constructed by Depauw in \cite{Depauw} belongs to $L_{loc}^1((0,+\infty);BV_{loc}(\R^2;\R^2))$. The following variation on Depauw's non-uniqueness result shows that Theorem \ref{thm_uniqueness} provides a non-trivial selection for bounded initial data, { even amongst solutions which have the same sign as the initial datum.}
\begin{theorem}\label{thm_depauw} 
	Consider an initial datum $\bar\rho\in L^\infty(\R^d)$. Then, there exists at least two bounded weak solutions of the Cauchy problem \eqref{eqPDE} along $\bb_{DP}$. {Moreover, if $\bar\rho$ is non-negative and $\bar\rho\neq 0$, then there exists at least two non-negative bounded weak solutions of the Cauchy problem \eqref{eqPDE} along $\bb_{DP}$.}
\end{theorem}
{\begin{remark} Given an arbitrary bounded initial datum, a bounded weak solution of \eqref{eqPDE} along $\bb_{DP}$ always exists by weak compactness. Adding the non-zero solution with zero initial datum given by Depauw constructs another bounded weak solution with the same initial datum. This was pointed out by one of the anonymous referees. However, even if the initial datum is non-negative, the second solution constructed by this procedure need not be non-negative. Some additional work is required to prove that, when the initial datum is non-negative, there exist two non-negative bounded weak solutions.
	\end{remark}}
\begin{remark}
	It would be interesting to establish a characterising property intrinsic to the weak solution of \eqref{eqPDE} along $\bb_{DP}$ selected by Theorem \ref{thm_uniqueness}. We moreover expect that the weak solutions constructed in Theorem \ref{thm_depauw} are distinct from the weak solution selected by Theorem \ref{thm_uniqueness}. 
\end{remark}

\subsection{Outline of ideas} 
\subsubsection{Ideas for Theorem \ref{thm_uniqueness}}
The convergence of weak solutions of \eqref{eqPDE} along a regularisation $(\bb^k)_{k\in\N}$ by convolution will be controlled by the convergence of solutions to a family backwards problem along $(\bb^k)_{k\in\N}$ with final datum given by a test function. The convergence of the solutions to the backwards problem will be established in Lemma \ref{lem_uniquenss_weak} and shown to be \emph{pointwise in time}. In order to get uniform in $k$ control of the Jacobian of the flow along $\bb^k$, we will crucially use that regularisation by convolution preserves the divergence-free structure. This will conclude well-posedness of the Cauchy problem \eqref{eqPDE} along the regularisation class $\mathcal{R}_{conv}$. 
\subsubsection{Ideas for Theorem \ref{thm_depauw}}
The classical construction of the vector field of Depauw $\bb_{DP}$ gives two non-unique bounded weak solutions $\zeta_1$ and $\zeta_2$ of \eqref{eqPDE} along $\bb_{DP}$. 
We will observe that any initial datum $\bar\rho$ can be weakly approximated by a sequence $\bar\rho_1^k$ localised on $\zeta_1(2^{-k-1},\cdot)$ and by a sequence $\bar\rho_2^k$ localised on $\zeta_2(2^{-k-1},\cdot)$. Correspondingly, we have unique bounded weak solutions $\rho_1^k$ and $\rho_2^k$ of \eqref{eqPDE} along $\bb_{DP}$ truncated up to time $2^{-k-1}$ and with initial datum $\bar\rho_1^k$ and $\bar\rho_2^k$ respectively thanks to Theorem \ref{thm_uniqueness_bv}. Weak limit points of the sequences $(\rho_1^k)_{k\in\N}$ and $(\rho_2^k)_{k\in\N}$ are then proven to be distinct bounded weak solutions of \eqref{eqPDE} along $\bb_{DP}$ with initial datum $\bar\rho$. 
\subsection{Plan of the paper} 
In Section \ref{section_boundary_value}, we introduce a boundary value problem. We prove that bounded weak solutions of this boundary value problem have a unique representative in $C(\R;w^*-L^\infty(\R^d))$. In Section \ref{section_theorem_uniq}, we prove Theorem \ref{thm_uniqueness} by using the work on the boundary value problem. In Section \ref{section_proof_depauw}, we give the classical construction of the vector field of Depauw \cite{Depauw} and record some properties of it. We then prove Theorem \ref{thm_depauw}. In the Appendix, we prove Proposition \ref{prop_well_posed}. 
\subsection*{Acknowledgements}
The author has been supported by the Simons foundation grant ID: 651475.
The author is thankful to his advisor Nikolay Tzvetkov for discussions, which have in particular inspired Definition \ref{defn_well_posedness_reg}, for his support, and for comments which have improved this manuscript.  The author is thankful to Lucas Huysmans for discussions and for pointing out a mistake. The author acknowledges the hospitality of the Pitcho Centre for Scientific Studies where this work was partly done. {The author thanks the anonymous referees for numerous suggestions, which have improved this paper.}

\bigskip 
\section{The boundary value problem} \label{section_boundary_value}
\subsection{Definitions}
 For $s\in \R$, consider the boundary value problem posed on $\R\times\R^d$,
\begin{equation}\label{eqn_boundary_value} \tag{BVP}
\left\{
\begin{split} 
\partial_t \rho +\div_x(\bb \rho)&=0 ,\\
\rho(s,x)&=\bar\rho(x). 
\end{split} 
\right. 
\end{equation}
We will work with bounded weak solutions of \eqref{eqn_boundary_value}. 
\begin{definition}\label{def_boundary_value}
	Consider a bounded vector field $\bb:\R\times\R^d\to\R^d$, a boundary datum $\bar\rho\in L^\infty(\R^d)$, and a closed interval $I\subset \R$.
	We say that $\rho\in L^\infty(I\times\R^d)$ is a bounded weak solution on $I$ to \eqref{eqn_boundary_value} along $\bb$, if for every $\phi \in C_c^\infty(I\times\mathbb{R}^d)$
	\begin{equation*}
\int_s^{+\infty}\int_{\mathbb{R}^d} \rho\left(\frac{\partial \phi}{\partial t} + {\bb}\cdot\nabla_x \phi\right) \; dxdt=-\int_{-\infty}^s\int_{\mathbb{R}^d} \rho\left(\frac{\partial \phi}{\partial t} + {\bb}\cdot\nabla_x \phi\right) \; dxdt = - \int_{\mathbb{R}^d} \bar\rho(x) \phi(s,x) \; dx.
	\end{equation*}
\end{definition}
For bounded weak solutions of the boundary value problem \eqref{eqn_boundary_value} on $\R$, we omit to specify the time interval. We record the following existence theorem for \eqref{eqn_boundary_value}. 

\begin{theorem}\label{thm_existence_boundary_value}
	Consider a bounded, divergence-free vector field $\bb:\R\times\R^d\to\R^d$, and an intial datum $\bar\rho\in L^\infty(\R^d)$. Then, there exists a bounded weak solution $\rho$ of \eqref{eqn_boundary_value} along $\bb$ satisfying $\|\rho\|_{L^\infty_{t,x}}\leq \|\bar\rho\|_{L^\infty_x}$. 
\end{theorem}
\begin{proof} 
 Let $\theta\in C^\infty_c(\R\times\R^d)$ be a standard mollifier, and let $\bb^k=\bb\star\theta^k$ be a regularisation of $\bb$. Let $\XX^k$ be the unique flow of $\bb^k$. Then, the unique weak solution $\rho^k$ of \eqref{eqn_boundary_value} along $\bb^k$ is given by : 
\begin{equation*}
\rho^k(t,x)\Leb{d}=\XX^k(t,s,x)_\#\bar\rho\Leb{d}.
\end{equation*}
Moreover, $\|\rho^k\|_{L^\infty_{t,x}}=\|\bar\rho\|_{L^\infty_x}$ because $\XX^k(t,s,\cdot)_\#\Leb{d}=\Leb{d}$. Therefore, by the Banach-Alaoglou Theorem, there is an increasing map $\psi:\N\to\N$ such that  $\rho^{\psi(k)}$ converges weak-star in $L^\infty(\R\times\R^d)$ to some $\rho$ as $k\to+\infty$, and $\|\rho\|_{L^\infty_{t.x}}\leq \|\bar\rho\|_{L^\infty}$ by weak lower semicontinuity of the norm. Since $\bb^k\to\bb$ strongly in $L^1_{loc}$, it follows that $\rho$ is a bounded weak solution of \eqref{eqn_a_priori_bound} along $\bb$. 
\end{proof} 

\subsection{Time continuous representative}
It is a standard fact that a bounded weak solution of \eqref{eqn_boundary_value}, although only in { $L^\infty(\R\times\R^d)$} always has a representative in $C(\R;w^*-L^\infty(\R^d))$, so we can take traces in time. This is recorded in the following lemma. 
\begin{lemma}\label{lem_given_weak_sol_exist_cont_repr}
	Consider a bounded vector field $\bb:\R\times\R^d\to\R^d$, a boundary datum $\bar\rho\in L^\infty(\R^d)$, a bounded weak solution $\rho$ of \eqref{eqn_boundary_value} along $\bb$, and a compact time interval $I\subset \R$ with non-empty interior. Then, $\rho$ admits a unique representative $\tilde\rho$ in $C(\R;w^*-L^\infty(\R^d))$ for which it holds that :
	\begin{enumerate} 
		\item $\sup_{t\in I}\|\tilde\rho(t,\cdot)\|_{L^\infty_x}=\|\rho\|_{L^\infty(I;L^\infty_x)};$
	\item for every $\phi\in C^1_c(\R^d)$, there exists a real constant $L_{\phi,I}>0$ such that for a.e. $t\in I$
		\begin{equation}\label{eqn_a_priori_bound}
	\Big|\frac{d}{dt}\int_{\R^d}\tilde\rho(t,x)\phi(x)dx\Big|\leq L_{\phi,I}\|\bb\|_{L^\infty(I\times\R^d)},
	\end{equation}
	in particular, if $\|\bb\|_{L^\infty(I\times\R^d)}=0$, then $\tilde\rho(t,\cdot)=\tilde\rho(s,\cdot)$ for every $t,s\in I$. 
	\end{enumerate} 
\end{lemma}
We shall call the unique representative in $C(\R;w^*-L^\infty(\R^d))$ of a bounded weak solution $\rho$ of \eqref{eqn_boundary_value}, \emph{the time continuous representative} of $\rho$.
\begin{proof}

	STEP 1 (A-priori bounds): Choose a representative of $\rho$ in $L^\infty(\R\times\R^d)$. As $\rho$ is a bounded weak solution of \eqref{eqn_boundary_value} along $\bb$, it holds that in $L^1_{loc}(\R)$:
	\begin{equation}
	\frac{d}{dt}\int_{\R^d}\rho(t,x)\phi(x)dx=\int_{\R^d}\rho(t,x)\bb(t,x)\cdot \nabla_x\phi(x)dx.
	\end{equation}
	Let $J\subset \R$ be a compact time interval.
	For every $\phi\in C^1_c(\R^d)$, there exists a real constant $L_{\phi,J}>0$ such that for a.e. $t\in J$ 
	\begin{equation}\label{eqn_a_priori_bound_comput} 
	\begin{split} 
	\Big|\frac{d}{dt}\int_{\R^d}\rho(t,x)\phi(x)dx\Big|&\leq\int_{\R^d}\big|\rho(t,x)\bb(t,x)\cdot\nabla_x\phi(x)\big|dx,\\
	&\leq \|\nabla_x\phi\|_{L_x^1}\|\rho\|_{L^\infty(J;L_x^\infty)}\|\bb\|_{L^\infty(J\times\R^d)}=L_{\phi,J}\|\bb\|_{L^\infty(J\times\R^d)}\\
	\end{split} 
	\end{equation}
	Therefore, the function $$\R\ni t\mapsto \int_{\R^d}\rho(t,x)\phi(x)dx,$$  admits a Lipschitz continuous representative with Lipschitz constant bounded by $L_{\phi,J}$ on $J$ by \cite[Theorem 8.2]{BrezisFunctAnal11} . 
	
		Consider the set $S_\phi\subset \R$, on which the above function coincides with its Lipschitz continuous representative. Note that $S_\phi$  depends on which representative of $\rho$ is chosen, and that $S_\phi$ has full measure by \cite[Theorem 8.2]{BrezisFunctAnal11}. 	Let $\mathcal{N}$ be a countable subset of $C_c^1(\R^d)$ dense in $L^1(\R^d)$, which exists by separability. Define $$S_\mathcal{N}=\bigcap_{\phi\in\mathcal{N}}S_\phi.$$ Notice that this set has full measure in $\R$, and is therefore also dense in $\R$. 
	For every $\phi\in\mathcal{N}$, define the function $f_\phi:\R\to\R$ as the unique continuous extension to $\R$ of the function defined for $t\in \mathcal{S}_\mathcal{N}$ by 
	\begin{equation} \label{eqn_fk_phi_on_S_N}
	f_\phi(t)=\int_{\R^d}\rho(t,x)\phi(x)dx.
	\end{equation} 
	We then have for every $\phi\in \mathcal{N}$
	\begin{equation}\label{eqn_unif_bounds} 
	\sup_{ t\in J}|f_\phi(t)|\leq\|\rho\|_{L^\infty(J;L_x^\infty)}\|\phi\|_{L_x^1}. 
	\end{equation}

	\bigskip 
	STEP 2 (Choosing a representative): 
	We now seek a representative $\tilde\rho$ in $C(\R;w^*-L^\infty(\R^d))$ of $\rho$. 
	Equip $C(\R;\R)$ with the topology of uniform convergence on compact time intervals, and equip $\mathcal{N}$ with the topology induced from $L^1(\R^d)$. 
	Consider the linear operator 
	\begin{equation}
	\Lambda:\mathcal{N}\ni \phi \mapsto f_\phi\in C(\R;\R). 
	\end{equation} 
	Note that by \eqref{eqn_unif_bounds}, it is a continuous linear operator.
By density of $\mathcal{N}$ in $L^1(\R^d)$ and by \eqref{eqn_unif_bounds}, $\Lambda$ admits a unique continuous extension $\tilde\Lambda$ to $L^1(\R^d)$, which satisfies 
\begin{equation}\label{eqn_tilde_lambda_norm_compact_time}
\sup_{t\in J}\big|[\tilde\Lambda(\phi)](t)\big|\leq \|\rho\|_{L^\infty (J;L^\infty_x)}\|\phi\|_{L^1_x}.
\end{equation}
 Equipping $(L^1(\R^d))^*$  with the weak-star topology, we can then identify $\tilde\Lambda$ with a family of linear functionals $$\tilde\Lambda: \R\ni t \mapsto \tilde\Lambda_t\in (L^1(\R^d))^*,$$ 
 such that for every $t\in\R$ $$\tilde\Lambda_t:L^1(\R^d)\ni\phi\mapsto[\tilde\Lambda(\phi)](t)\in \R.$$ 
 Thus, by the duality $(L^1(\R^d))^*\cong L^\infty(\R^d)$, for every $t\in\R$ we have that $\tilde\Lambda_t$ can be uniquely identified with an element $\tilde\rho(t,\cdot)$ in $L^\infty(\R^d)$.  
 Let us check that $\tilde\rho$ is in $C(\R;w^*-L^\infty(\R^d))$. The weak-star topology on $L^\infty(\R^d)$ is induced by the family of seminorms defined by $p_\phi(f)=|\int_{\R^d}\phi(x)f(x)dx|$ for every $\phi \in L^1(\R^d)$ and every $f\in L^\infty(\R^d)$. For every $\phi\in L^1(\R^d)$ and every $t_1,t_2\in\R$, we then have 
 \begin{equation}
 \begin{split} 
 \Big|p_\phi\Big(\tilde\rho(t_1,\cdot)-\tilde\rho(t_2,\cdot)\Big)\Big|&=\Big|\int_{\R^d}\phi(x)\tilde\rho(t_1,x)dx-\int_{\R^d}\phi(x)\tilde\rho(t_2,x)dx\Big|\\
 &=\Big|\tilde\Lambda_{t_1}(\phi)-\tilde\Lambda_{t_2}(\phi)\Big|\\&=\Big|[\tilde\Lambda(\phi)](t_1)-[\tilde\Lambda(\phi)](t_2)\Big|,
 \end{split} 
 \end{equation}
 which can be made arbitrarly small by continuity of $\tilde\Lambda(\phi)$ and by taking $|t_1-t_2|$ sufficiently small.
Therefore, $\tilde\rho \in C(\R;w^*-L^\infty(\R^d))$, and
 \begin{equation}\label{eqn_lambda_rho_1}
 \tilde\Lambda_t(\phi)=\int_{\R^d}\tilde\rho(t,x)\phi(x)dx\qquad\forall\phi\in L^1(\R^d).
 \end{equation}
 Notice also that by construction, $\tilde\Lambda_t$ is the unique continuous extension of the linear functional $$\Lambda_t:\mathcal{N}\ni\phi\mapsto f_\phi(t)\in\R.$$
So, for every $t\in \mathcal{S}_\mathcal{N}$ we have 
	\begin{equation}
	\tilde\Lambda_t(\phi)=\int_{\R^d}\rho(t,x)\phi(x)dx \qquad\forall \phi\in \mathcal{N}. 
	\end{equation}
Together with \eqref{eqn_lambda_rho_1}, the above equation implies by density of $\mathcal{N}$ in $L^1(\R^d)$ that for every $t\in \mathcal{S}_\mathcal{N}$ it holds that $\tilde\rho(t,\cdot)=\rho(t,\cdot)$ as bounded functions on $\R^d$. Since $\mathcal{S}_\mathcal{N}$ has full measure in $\R$, we have $\tilde\rho(t,x)=\rho(t,x)$ for $\Leb{d+1}$-a.e. $(t,x)\in\R\times\R^d$, whence $\tilde\rho$ is a bounded weak solution of \eqref{eqPDE} along $\bb$. Item $(i)$ then follows by \eqref{eqn_tilde_lambda_norm_compact_time} for the time interval $I$. Item $(ii)$ follows by making a-priori bounds on $\tilde\rho$ as in \eqref{eqn_a_priori_bound_comput}.

	\bigskip 
\end{proof}

\subsection{Convergence of time continuous representatives} 
Solutions of \eqref{eqn_boundary_value} along a suitable regularisation of a bounded, divergence-free $\bb$ have their time continuous representative converging $C(\R;w^*-L^\infty(\R^d))$. This is recorded in the following lemma. 
\begin{lemma}\label{lem_existence}
	Consider a bounded, divergence-free vector field $\bb:\R\times\R^d\to \R^d$, and boundary datum $\bar\rho\in  L^\infty(\R^d)$, and a regularisation $(\bb^k)_{k\in\N}$ such that $\div_x\bb^k=0$. Assume that the unique bounded weak solution $\rho^k$ of \eqref{eqn_boundary_value} along $\bb^k$ converges weakly-star in $L^\infty(\R\times\R^d)$ to some bounded weak solution $\rho$ of \eqref{eqn_boundary_value} along $\bb$. 
	Then, the time continuous representative of $\rho^k$ converges in $C(\R;w^*-L^\infty(\R^d))$ to the time continuous representative of $\rho$. 
\end{lemma}
\begin{proof}
	Let $\XX^k$ be the unique flow of $\bb^k$. Then, the unique weak solution $\rho^k$ of \eqref{eqn_boundary_value} along $\bb^k$ is given by : 
	\begin{equation*}
	\rho^k(t,x)\Leb{d}=\XX^k(t,s,x)_\#\bar\rho\Leb{d},
	\end{equation*}
	and by Lemma \ref{lem_given_weak_sol_exist_cont_repr}, has a unique time continuous representative $\tilde\rho^k$. Similarly, $\rho$ has a unique time continuous representative $\tilde\rho$. 
	
We will use a subsubsequence arguement to show that $(\tilde\rho^k)_{k\in\N}$ has a unique accumulation point. Accordingly, let $\xi:\N\to\N$ be an increasing map. We will show that there is another increasing map $\delta:\N\to\N$ such that $\tilde\rho^{(\delta\circ\xi)(k)}$ converges in $C(\R;w^*-L^\infty(\R^d))$ to the time continuous representative of $\tilde\rho$ as $k\to+\infty$. 
	
	\bigskip
		STEP 1 (A-priori bounds): Since $\bb$ is divergence-free, then $\bb^k$ is also divergence-free, and we have $\XX^k(t,s,\cdot)_\#\Leb{d}=\Leb{d}$, whence $\|\rho^k\|_{L^\infty_{t,x}}=\|\bar\rho\|_{L^\infty_x}$. Therefore, by item $(i)$ of Lemma \ref{lem_given_weak_sol_exist_cont_repr} and taking a covering of $\R$ by compact intervals, we get
   		\begin{equation}\label{eqn_control_L^infty}
   		\sup_{t\in\R}	\|\tilde\rho^k(t,\cdot)\|_{L^\infty_{x}}=\|\bar\rho\|_{L^\infty_x}.
   			\end{equation} 
	By the a-priori bound \eqref{eqn_a_priori_bound}, for every $\phi\in C^1_c(\R^d)$, and every compact time interval $I\subset \R$, there exists a real constant $L_{\phi,I}>0$ such that 
the functions 
\begin{equation*}
f_\phi^k:\R\ni t\mapsto \int_{\R^d} \rho^k(t,x)\phi(x)dx,
\end{equation*}
 are continuous with Lipschitz constant on $I$ bounded by $L_{\phi,I}$. 
 and satisfy 
 $$\sup_{t\in\R,k\in\N}|f^k_\phi(t)|\leq  \|\bar\rho\|_{L_x^\infty}\|\phi\|_{L_x^1}.$$

\bigskip 
		STEP 2 (Compactness):
		Let $\mathcal{N}\subset C^1_c(\R^d)$ be a countable, dense subset of $L^1(\R^d)$. 
		By Ascoli's Theorem, for every $\phi\in \mathcal{N}$, there is an increasing map $\psi:\N\to \N$ such that $f^{(\psi\circ\xi)(k)}_\phi$ converges uniformly on compact time intervals as $k\to+\infty$. Therefore, there exists a diagonal increasing map $\delta:\N\to\N$ such that 
		\begin{itemize} 
		\item[(C)]\label{item_conv} for every $\phi\in \mathcal{N}$, the function $f^{(\delta\circ\xi)(k)}_\phi$ converge uniformly on compact time intervals to a continuous function $f_\phi$ as $k\to+\infty$. 
	\end{itemize}

\bigskip
		STEP 3 (Choosing a representative): We reason as in Step 2 of the proof of Lemma \ref{lem_given_weak_sol_exist_cont_repr}.
		Equip $C(\R;\R)$ with the topology of uniform convergence on compact time intervals, and equip $\mathcal{N}$ with the topology induced from $L^1(\R^d)$. 
		Consider the continuous linear operator 
		\begin{equation}
		\Lambda:\mathcal{N}\ni \phi \mapsto f_\phi\in C(\R;\R). 
		\end{equation} 
		By density of $\mathcal{N}$ in $L^1(\R^d)$ and by \eqref{eqn_control_L^infty}, $\Lambda$ admits a unique continuous extension $\tilde\Lambda$ to $L^1(\R^d)$. Equipping $(L^1(\R^d))^*$  with the weak-star topology, we can then identify $\tilde\Lambda$ with a family of linear functionals $$\tilde\Lambda: \R\ni t \mapsto \Lambda_t\in (L^1(\R^d))^*,$$ 
		such that for every $t\in\R$ $$\tilde\Lambda_t:L^1(\R^d)\ni\phi\mapsto[\tilde\Lambda(\phi)](t)\in \R.$$ 
	Thus, by the duality $(L^1(\R^d))^*\cong L^\infty(\R^d)$, for every $t\in\R$ we have that $\tilde\Lambda_t$ can be uniquely identified with an element $\zeta(t,\cdot)$ in $L^\infty(\R^d)$.  
	Let us check that $\zeta$ is in $C(\R;w^*-L^\infty(\R^d))$. The weak-star topology on $L^\infty(\R^d)$ is induced by the family of seminorms defined by $p_\phi(f)=|\int_{\R^d}\phi(x)f(x)dx|$ for every $\phi \in L^1(\R^d)$ and every $f\in L^\infty(\R^d)$. For every $\phi\in L^1(\R^d)$ and every $t_1,t_2\in \R$, we then have 
	\begin{equation}
	\begin{split} 
	\Big|p_\phi\Big(\zeta(t_1,\cdot)-\zeta(t_2,\cdot)\Big)\Big|&=\Big|\int_{\R^d}\phi(x)\zeta(t_1,x)dx-\int_{\R^d}\phi(x)\zeta(t_2,x)dx\Big|,\\
	&=\Big|\tilde\Lambda_{t_1}(\phi)-\tilde\Lambda_{t_2}(\phi)\Big|,\\&=\Big|[\tilde\Lambda(\phi)](t_1)-[\tilde\Lambda(\phi)](t_2)\Big|,
	\end{split} 
	\end{equation}
	which can be made arbitrarly small by continuity of $\tilde\Lambda(\phi)$ and by taking $|t_1-t_2|$ sufficiently small.
	Therefore, $\zeta \in C(\R;w^*-L^\infty(\R^d))$, and
	\begin{equation}\label{eqn_lambda_rho}
	\tilde\Lambda_t(\phi)=\int_{\R^d}\zeta(t,x)\phi(x)dx\qquad\forall\phi\in L^1(\R^d).
	\end{equation}
	
	Notice also that by construction, $\tilde\Lambda_t$ is the unique continuous extension of the linear functional $$\Lambda_t:\mathcal{N}\ni\phi\mapsto f_\phi(t)\in\R.$$
	So, for every $t\in \mathcal{S}_\mathcal{N}$ we have 
	\begin{equation}
	\tilde\Lambda_t(\phi)=\int_{\R^d}\rho(t,x)\phi(x)dx \qquad\forall \phi\in \mathcal{N}. 
	\end{equation}
		
		\bigskip 
		STEP 4 (Identification of $\zeta=\tilde\rho$):
		Locally uniformly in $t\in\R$, 
		we have 
		\begin{equation}
		\lim_{k\to+\infty}\int_{\R^d}\tilde\rho^{(\delta\circ \xi)(k)}(t,x)\phi(x)dx=\int_{\R^d}\zeta(t,x)\phi(x)dx \qquad\forall\phi\in\mathcal{N}.
		\end{equation}
Therefore, for any $\psi\in C^1_c(\R)$ and any $\phi\in \mathcal{N}$, we have 
\begin{equation}\label{eqn_conv_against_product} 
	\lim_{k\to+\infty}\int_\R\int_{\R^d}\tilde\rho^{(\delta\circ \xi)(k)}(t,x)\phi(x)\psi(t)dxdt=\int_\R\int_{\R^d}\zeta(t,x)\phi(x)\psi(t)dxdt.
\end{equation}
Let $\e>0$. 
By density in $L^1(\R \times\R^d)$ of the set
\begin{equation*}
\mathcal{D}:=\Big\{\sum_{k=1}^N\phi_k(x)\psi_k(t) \;:\;\phi_k\in\mathcal{N}, \; \psi_k\in C^1_c(\R),\;N\in\N\Big\}, 
\end{equation*}
for any $\eta\in L^1(\R\times\R^d)$, there is $\nu\in\mathcal{D}$ such that $\|\eta-\nu\|_{L_{t,x}^1}<\e$. 
Write 
\begin{equation}
\begin{split} 
\int_\R\int_{\R^d}\tilde\rho^{(\delta\circ \xi)(k)}(t,x)\eta(t,x)dxdt=&\int_\R\int_{\R^d}\tilde\rho^{(\delta\circ \xi)(k)}(t,x)(\eta(t,x)-\nu(t,x))dxdt\\
&+\int_\R\int_{\R^d}\tilde\rho^{(\delta\circ \xi)(k)}(t,x)\nu(t,x)dxdt.\\
\end{split} 
\end{equation}
Using \eqref{eqn_control_L^infty} and H\"older inequality, we have 
\begin{equation}
\Big|\int_\R\int_{\R^d}\tilde\rho^{(\delta\circ \xi)(k)}(t,x)(\eta(t,x)-\nu(t,x))dxdt\Big|\leq \|\bar\rho\|_{L^\infty_x}\|\eta-\nu\|_{L^1_{t,x}}\leq \e\|\bar\rho\|_{L^\infty_x}. 
\end{equation}
By \eqref{eqn_conv_against_product}, we have for $k$ sufficiently large that
\begin{equation}
\Big|\int_{\R}\int_{\R^d}\tilde\rho^{(\delta\circ\xi)(k)}(t,x)\nu(t,x)dxdt-\int_{\R}\int_{\R^d}\zeta(t,x)\nu(t,x)dxdt\Big|<\e, 
\end{equation}
and as $\|\zeta\|_{L^\infty_{t,x}}\leq \|\bar\rho\|_{L^\infty_x}$, we also have that 
\begin{equation}
\Big|\int_{\R}\int_{\R^d}\zeta(t,x)(\nu(t,x)-\eta(t,x))dxdt\Big|\leq\e \|\bar\rho\|_{L^\infty_x},
\end{equation}
therefore we have 
\begin{equation}
\lim_{k\to+\infty}\int_{\R}\int_{\R^d}\tilde\rho^{(\delta\circ \xi)(k)}(t,x)\eta(t,x)dxdt= \int_{\R}\int_{\R^d}\zeta(t,x)\eta(t,x)dxdt.
\end{equation}
 So, 
		it holds that $\rho^{(\delta\circ\xi)(k)}$ converges weakly-star in $L^\infty(\R\times\R^d)$ to $\zeta$ as $k\to+\infty$. 
		By the hypothesis of the lemma, we have $\zeta=\rho$ as functions in $L^\infty(\R\times\R^d)$. By uniqueness of the time continuous representative $\tilde\rho$ of $\rho$, we have $\zeta=\tilde\rho$ as functions in $C(\R;w^*-L^\infty(\R^d))$.

		\bigskip 
		STEP 5 (Convergence to $\tilde\rho$): We now prove that $\tilde\rho^{(\delta\circ\xi)(k)}$ converges to $ \tilde\rho$ in $C(\R;w^*-L^\infty(\R^d))$ as $k\to+\infty$. Let $\phi\in L^1(\R^d)$, $\tilde\phi\in \mathcal{N}$, and $\e>0$ such that $\|\tilde\phi-\phi\|_{L^1_x}<\e$. Note that, 
		\begin{equation}
		\begin{split} 
	\int_{\R^d}\tilde\rho^{(\delta\circ\xi)(k)}(t,x)\phi(x)dx&=\int_{\R^d}\tilde\rho^{(\delta\circ\xi)(k)}(t,x)\Big(\phi(x)-\tilde\phi(x)\Big)dx+\int_{\R^d}\Big(\tilde\rho^{(\delta\circ\xi)(k)}(t,x)-\tilde\rho(t,x)\Big)\tilde\phi(x)dx\\
		&+\int_{\R^d}\tilde\rho(t,x)\Big(\tilde\phi(x)-\phi(x)\Big)dx+\int_{\R^d}\rho(t,x)\phi(x)dx.
		\end{split} 
		\end{equation}
		The first term is bounded by $\e\|\bar\rho\|_{L^\infty_x}$ by H\"older inequality and by \eqref{eqn_control_L^infty}. The second term is bounded by $\e$ for $k$ sufficiently large thanks to (C). The third term is bounded by $\e\|\bar\rho\|_{L^\infty_x}$ by H\"older inequality and by \eqref{eqn_control_L^infty}.
		Therefore, we see that 
		\begin{equation*}
		\lim_{k\to+\infty}\int_{\R^d}\tilde\rho^{(\delta\circ\xi)(k)}(t,x)\phi(x)dx=\int_{\R^d}\tilde\rho(t,x)\phi(x)dx,
		\end{equation*}
		whence, $\tilde\rho^{(\delta\circ\xi)(k)}$ converges to $\tilde\rho$ in $C(\R;w^*-L^\infty(\R^d))$ as $k\to+\infty$. Since $\xi:\N\to\N$ was an arbitrary increasing map, by the subsubsequence lemma, it follows that the whole sequence $\tilde\rho^k$ converges in $C(\R;w^*-L^\infty(\R^d))$ to the time continuous representative of $\tilde\rho$ as $k\to+\infty$, which proves the thesis. 
 \end{proof} 

\subsection{Uniqueness for $\bb\in L^1_{loc}(\R;BV_{loc})$}
We will use the following straightforward consequence of Theorem \ref{thm_uniqueness_bv} for the boundary value problem \eqref{eqn_boundary_value}. 
	\begin{theorem}\label{thm_ambrosio}
Consider a bounded, divergence-free vector field $\bb:\R\times\R^d\to\R^d$and a boundary datum $\bar\rho\in L^\infty(\R^d)$. Assume that $\bb\in  L_{loc}^1(\R;{\rm BV_{loc}}(\R^d;\R^d))$. 
 Then, there exists a unique bounded weak solution of \eqref{eqn_boundary_value} along $\bb$. 
	\end{theorem}
\begin{proof}
{	Recall that $s\in\R$ is the time at which the boundary datum is imposed in \eqref{eqn_boundary_value}.}
Let $\rho$ be a bounded weak solution of \eqref{eqn_boundary_value} whose existence follows from Theorem \ref{thm_existence_boundary_value}. By Lemma \ref{lem_given_weak_sol_exist_cont_repr}, $\rho$ has a unique time continuous representative $\tilde\rho$. 
	Define the vector fields $\bb^\leq(t,x)=\bb(s-t,x)$ and $\bb^\geq (t,x)=\bb(t+s,x)$. Then, for $t\in(0,+\infty)$ $\rho^\leq(t,x)=\rho(s-t,x)$ and $\rho^\geq (t,x)=\rho(t+s,x)$ are bounded weak solutions of \eqref{eqPDE} along $\bb^\leq$ and $\bb^\geq$ respectively, and they are thus uniquely determined for $\Leb{d+1}$-a.e. $(t,x)\in (0,+\infty)\times\R^d$ by Theorem \ref{thm_uniqueness_bv}. Therefore $\rho(t,x)$ is uniquely determined for $\Leb{d+1}$-a.e. $(t,x)\in \R\times\R^d$. This proves the thesis.
\end{proof}

\bigskip 
\section{Proof of Theorem \ref{thm_uniqueness}}\label{section_theorem_uniq}
\subsection{Backward uniqueness for $\bb\in L^1_{loc}((0,+\infty);BV_{loc})$}
The following crucial lemma uniquely identifies bounded weak solutions for a boundary datum at a time $s>0$. 
\begin{lemma} \label{lem_uniquenss_weak} 
Consider a bounded, divergence-free vector field $\bb:\R\times\R^d\to \R^d$, a boundary datum $\bar\rho\in L^\infty(\R^d)$, and $s>0$. Assume that $\bb\in L_{loc}^1((0,+\infty);{\rm BV_{loc}}(\R^d;\R^d))$, and that $\bb(t,\cdot)\equiv 0$ for $t<0$. 
Then, there exists a unique bounded weak solution $\rho_{\bar\rho,s}$ to the boundary value problem \eqref{eqn_boundary_value}.

Moreover, for every standard mollifier $\theta\in C^\infty_c(\R\times\R^d)$, the time continuous representative $\tilde\rho_{\bar\rho,s}^k$ of the unique bounded weak solution  of \eqref{eqn_boundary_value} along $\bb\star \theta^k$ converges in $C(\R;w^*-L^\infty(\R^d))$ to the time continuous representative of $\rho_{\bar\rho,s}$. 
\end{lemma} 
\begin{proof}
	
Let $\rho$ be a bounded weak solution of \eqref{eqn_boundary_value} along $\bb$ whose existence follows from Theorem \ref{thm_existence_boundary_value}.  By Lemma \ref{lem_given_weak_sol_exist_cont_repr}, $\rho$ has a unique time continuous representative $\tilde\rho$. 
	To prove uniqueness, we will show that $\tilde\rho$ is uniquely determined by bounded weak solutions of \eqref{eqn_boundary_value} along a regularised version of $\bb$. Let us consider for any $\tau>0$ the vector fields 
	\begin{equation*}
	\bb_\tau(t,x):=\left\{ 
	\begin{split} 
	\bb(t,x)\qquad &\text{if} \quad t\geq \tau,\\
	0\qquad &\text{if}\quad t<\tau.
	\end{split}
	\right. 
	\end{equation*}

STEP 1 (Determining $\tilde\rho(t,\cdot)$ for $t\in (0,+\infty)$):
It follows directly from Definition \ref{def_boundary_value} that $\tilde\rho$ is a bounded weak solution on $[\tau,+\infty)$ of \eqref{eqn_boundary_value} along $\bb_\tau$. Also, notice that	 $\bb_\tau(t,x)=\bb(t,x)$ for $\Leb{d+1}$-a.e. $(t,x)\in(\tau,+\infty)\times\R^d$. As $\bb_\tau$ satisfies the hypothesis of Theorem \ref{thm_ambrosio}, there exists a unique bounded weak solution
$\rho_{\phi,\tau,s}$ of 
\eqref{eqn_extension_zero} along $\bb_\tau$ with time continuous representative $\tilde\rho_{\phi,\tau,s}$. Therefore, for every $t\in [\tau,+\infty)$, we have for a.e. $x\in\R^d$
\begin{equation}\label{eqn_determine_weak}
\tilde\rho_{\phi,\tau,s}(t,x)=\tilde\rho(t,x). 
\end{equation}
  As $\tau>0$ is an arbitrary positive real number, \eqref{eqn_determine_weak} uniquely determines $\tilde\rho$ in $C((0,+\infty); w^*-L^\infty(\R^d))$. 
  
  \bigskip 
  STEP 2 (Determining $\tilde\rho(0,\cdot)$):
{Since $\tilde\rho$ belongs to $C(\R;w^*-L^\infty(\R^d))$, we know that  $\tilde\rho(0,\cdot)=\lim_{t\downarrow 0}\tilde\rho(t,\cdot)$, where the limit is taken with respect to the weak-star topology on $L^\infty(\R^d)$. As $\tilde\rho$ has been uniquely determined for positive times in Step 1, it follows that $\tilde\rho(0,\cdot)$ is also uniquely determined. }

 \bigskip
 STEP 3 (Determining $\tilde\rho(t,\cdot)$ for $t<0$):
  Since $\bb(t,\cdot)\equiv0$ for $t<0$, by item $(ii)$ of Lemma \ref{lem_given_weak_sol_exist_cont_repr}, we have for $t<0$
  \begin{equation}
  \int_{\R^d}\tilde\rho(t,x)\phi(x)dx=  \int_{\R^d}\tilde\rho(0,x)\phi(x)dx\qquad\forall\phi\in C^1_c(\R^d). 
  \end{equation}
By density of $C_c^1(\R^d)$ in $L^1(\R^d)$, it follows that $\tilde\rho(t,\cdot)=\tilde\rho(0,\cdot)$ as bounded functions. Therefore $\tilde\rho(t,\cdot)$ is uniquely determined for every $t\in\R$. By uniqueness of the time continuous representative of a bounded weak solution, it follows that there exists a unique bounded weak solution $\rho$ of \eqref{eqn_a_priori_bound} along $\bb$, which we shall denote by $\rho_{\bar\rho,s}$.

	\bigskip
	STEP 4 (Convergence of regularisations):
	Let us come to the final part of the statement. Let $\theta\in C_c^\infty(\R\times\R^d)$ be a standard mollifier. 
By the classical theory, there exists a unique bounded weak solution \eqref{eqn_boundary_value} along $\bb\star\theta^k$ we shall denote by $\rho^k_{\bar\rho,s}$ and satisfying $\|\rho^k_{\bar\rho,s}\|_{L^\infty_{t,x}}=\|\bar\rho\|_{L^\infty_x}$. Therefore, by the Banach-Alaouglu theorem, up to extracing a subsequence such that $\rho_{\bar\rho,s}^k$ converges to some $\rho$ weakly-star in $L^\infty(\R\times\R^d)$. Since $\bb^k\to \bb$ strongly in $L^1_{loc}$, it follows that $\rho$ is a bounded weak solution of \eqref{eqn_a_priori_bound} along $\bb$. The uniqueness in first part of this lemma implies that $\rho=\rho_{\bar\rho,s}$. A subsubsequence argument therefore implies that the whole sequence $\rho_{\bar\rho,s}^k$ converges to $\rho_{\bar\rho,s}$ weakly-star in $L^\infty(\R\times\R^d)$ as $k\to+\infty$.
By Lemma \ref{lem_existence}, the time continuous representative $\tilde\rho^k_{\bar\rho,s}$ converges in $C(\R;w^*-L^\infty(\R^d))$ to the time continuous representative $\tilde\rho_{\bar\rho,s}$. This proves the thesis. 
	
	
	
\end{proof}

\subsection{Proof of Theorem \ref{thm_uniqueness}}\label{sec_proof_thm_uniq}
We can now conclude this section. 
{Let us give some preliminary tools, which will be useful for establishing weak compactness. Fix a measure space $(X,\mathscr{E})$ and a positive measure $\mu$ on it. We can then form the Lebesgue spaces $L^p(X,\mu)$, which are Banach spaces for $1\leq p\leq +\infty$. Recall that a family $\mathscr{F}\subset L^1(X,\mu)$ is equi-integrable, if 
\begin{equation}
\lim_{C\to+\infty}\sup_{f\in\mathscr{F}}\int_{\{|f|>C\}}|f|d\mu=0.
\end{equation}
The Dunford-Pettis theorem (see \cite[Theorem 1.38]{AmbFuscPalBV}) asserts that $\mathscr{F}$ is relatively weakly sequentially compact subset of $L^1(X,\mu),$ if and only if $\mathscr{F}$ is equi-integrable. }
\begin{proof}[Proof of Theorem \ref{thm_uniqueness} ]
Extend $\bb$ by zero for negative times and choose a standard mollifier $\theta\in C_c^\infty(\R\times\R^d)$. Fix $\phi \in C^\infty_c((0,+\infty)\times\R^d)$, and write $\bb^k=\bb\star \theta^k$. {Recall that $\XX^k$ is the flow along $\bb^k$. As for every $k\in\N,$ we have that $$\sup_{(t,x)\in[0,+\infty)\times\R^d}|\bb^k(t,x)|\leq \|\bb\|_{L^\infty_{t,x}},$$
therefore the following finite speed of propagation property holds:
\begin{itemize}
	\item [(FS)] for every $x\in\R^d$, every $k\in\N$, and every $s,t\in \R$ we have $|\XX^k(t,s,x)-\XX^k(s,s,x)|\leq |t-s|\|\bb\|_{L^\infty_{t,x}}$. 
\end{itemize}
}
\bigskip 
STEP 1 (Solutions from the classical theory):
By the classical theory, we have 
\begin{enumerate} 
	{
	\item the density $\rho^k:[0,+\infty)\times\R^d\to \R$ given by the formula $$\rho^k(t,\cdot)\Leb{d}=\XX^k(t,0,\cdot)_\#\bar\rho\Leb{d},$$ is the unique weak solution of \eqref{eqPDE} along $\bb^k$ with initial datum $\bar\rho$.} 
\item the bounded density $\tilde\rho^k_{\phi(s,\cdot),s}:\R\times\R^d\to\R$ given by the formula $$\tilde\rho^k_{\phi(s,\cdot),s}(t,\cdot)\Leb{d}=\XX^k(t,s,\cdot)_\#\phi(s,\cdot)\Leb{d},$$ is the time continuous representative of the unique bounded weak solution of \eqref{eqn_boundary_value} along $\bb^k$ with boundary datum $\phi$, { and for every $s\in\R$ satisfies that $\sup_{t\in\R} \|\tilde\rho^k_{\phi(s,\cdot),s}(t,\cdot)\|_{L^\infty_x}\leq\|\phi\|_{L^\infty_{t,x}}$.}
{
	By (FS) above, and since $\phi$ is compactly supported, the following property holds:
	
\begin{itemize} 
\item[(CS)] there exists $R>0$ sufficiently large that $\supp \tilde\rho^k_{\phi(s,\cdot),s}(0,\cdot)\subset B_R(0)$ and$\supp\phi(s,\cdot)\subset B_R(0)$ for every $s\in (0,+\infty)$. 
\end{itemize} }
\end{enumerate}
We will use a subsubsequence argument to characterise the weak limit points of $\rho^k$. Let $\xi:\N\to\N$ be an increasing map.

\bigskip 
STEP 2 (Compactness):
{Let us check relative weak sequential compactness of the sequence $(\rho^{\xi(k)})_{k\in\N}$ in $L^1([0,n]\times [-n,n]^d)$ for every $n\in \N$. First, note that as $\bb^{\xi(k)}$ is divergence-free for every $k\in\N$, we have $\rho^{\xi(k)}(t,\cdot)\Leb{d}=\XX^{\xi(k)}(t,0,\cdot)_\#\bar\rho\Leb{d}=\bar\rho(\XX^{\xi(k)}(0,t,\cdot))\Leb{d}$ for every $t\geq 0$. Now, let $n\in\N$.
	Let us now consider the following quantity
\begin{equation*}
\begin{split} 
I:&=\lim_{C\to+\infty}\sup_{k\in\N}\int_0^n\int_{\{|\rho^{\xi(k)}(t,\cdot)|>C\}\cap  [-n,n]^d}|\rho^{\xi(k)}(t,x)|dxdt\\
&=\lim_{C\to+\infty}\sup_{k\in\N}\int_0^n\int_{\{|\rho^{\xi(k)}(t,\cdot)|>C\}\cap  [-n,n]^d}|\bar\rho(\XX^{\xi(k)}(0,t,\cdot))|dxdt\\
&=\lim_{C\to+\infty}\sup_{k\in\N}\int_0^n\int_{\XX^{\xi(k)}\big(0,t,\{|\rho^{\xi(k)}(t,\cdot)|>C\}\cap [-n,n]^d\big)}|\bar\rho(x)|dxdt
\end{split}
\end{equation*}
Observe that by (FS), we have $$\XX^{\xi(k)}(t,0,[-n,n]^d)\subset [-n-n\|\bb\|_{L^\infty_{t,x}},n+n\|\bb\|_{L^\infty_{t,x}}]^d$$ for every $t\in [0,n]$
whence, up to sets of vanishing Lebesgue measure, the following inclusions hold
\begin{equation} 
\begin{split} 
& \XX^{\xi(k)}\big(0,t,\{|\rho^{\xi(k)}(t,\cdot)|>C\}\cap  [-n,n]^d\big)\\
&\subset \XX^{\xi(k)}(0,t,\{|\rho^{\xi(k)}(t,\cdot)|>C\})\cap \XX^{\xi(k)}(0,t, [-n,n]^d)\\
 &\subset \XX^{\xi(k)}(0,t,\{|\rho^{\xi(k)}(t,\cdot)|>C\})\cap  [-n-n\|\bb\|_{L^\infty_{t,x}},n+n\|\bb\|_{L^\infty_{t,x}}]^d\\
 &=\{|\bar\rho|>C\}\cap [-n-n\|\bb\|_{L^\infty_{t,x}},n+n\|\bb\|_{L^\infty_{t,x}}]^d. 
 \end{split} 
 \end{equation} 
 Therefore,
 \begin{equation}
 \begin{split}
 I&\leq 
 n\lim_{C\to+\infty}\sup_{k\in\N}\int_{\{|\bar\rho|>C\}\cap[-n-n\|\bb\|_{L^\infty_{t,x}},n+n\|\bb\|_{L^\infty_{t,x}}]^d}|\bar\rho(x)|dx=0.
 \end{split} 
 \end{equation}
 As $n$ was arbitrary, by the Dunford-Pettis theorem (see begin of Section \ref{sec_proof_thm_uniq}), for every $n\in\N$, the sequence $(\rho^{\xi(k)})_{k\in\N}$ is relatively weakly sequentially compact in $L^1([0,n]\times [-n,n]^d)$. 
 
 By a diagonal argument, there is an increasing map $\psi:\N\to\N$ and $\rho\in L^1_{loc}([0,+\infty)\times\R^d)$ such that, for every $n\in\N$, we have that $\rho^{(\psi\circ\xi)(k)}$ converges weakly in $L^1([0,n]\times[-n,n]^d)$ as $k\to+\infty$. 

\bigskip
STEP 3 (Dual representation of $\rho^k$):
{
We thus have for every $k\in \N$
\begin{equation}\label{eqn_change_variable} 
\begin{split} 
\int_0^{+\infty}\int_{B_R(0)}\rho^k(s,x)\phi(s,x)dxds
&=\int_0^{+\infty} \int_{\R^d}\rho^k(s,x)\phi(s,x)dxds\\
&=\int_0^{+\infty}\int_{\R^d} \phi(s,x)\XX^k(s,0,\cdot)_\#\bar\rho(x)dxds\\
&=\int_0^{+\infty}\int_{\R^d} \bar\rho(x)\XX^k(0,s,\cdot)_\#\phi(s,x)dxds\\
&=\int_0^{+\infty}\int_{\R^d} \bar\rho(x) \tilde\rho^k_{\phi(s,\cdot),s}(0,x)dxds\\
&=\int_0^{+\infty}\int_{B_R(0)} \bar\rho(x) \tilde\rho^k_{\phi(s,\cdot),s}(0,x)dxds,
\end{split}
\end{equation}
where in the first and in the last equality, we have used (CS),}
 and where we have performed a change of variable in the third equality, and used that the Jacobian of $\XX^k(s,0,x)$ is one because $\div_x \bb^k=0$ for every $t\in\R$.

\bigskip 
STEP 4 (Passing to the limit):
By Lemma \ref{lem_uniquenss_weak}, for every $s>0$ the time continuous representative $\tilde\rho^k_{\phi(s,\cdot),s}$ converges  in $C(\R;w^*-L^\infty(\R^d))$ to the time continuous representative $\tilde\rho_{\phi(s,\cdot),s}$ of the unique bounded weak solution  of \eqref{eqn_boundary_value} along $\bb$. 
{
	As $\mathbb{1}_{B_R(0)}\bar\rho$ belongs to $L^1(\R^d)$, 
	we therefore have 
	\begin{equation*}
	\int_{B_R(0)} \bar\rho(x) \tilde\rho^k_{\phi(s,\cdot),s}(0,x)dx \underset{k\to+\infty}{\longrightarrow}\int_{B_R(0)} \bar\rho(x) \tilde\rho_{\phi(s,\cdot),s}(0,x)dx \qquad\forall s\geq 0.
	\end{equation*}
	In view of (CS), we have $\supp \tilde\rho_{\phi(s,\cdot),s}(0,\cdot)\subset B_R(0)$, and for $T>0$ sufficiently large, we have $\tilde\rho_{\phi(s,\cdot),s}(0,\cdot)\equiv 0$ for every $s\geq T$. 
Therefore, we have the following uniform in $k$ bound
\begin{equation*}
\begin{split} 
\int_0^{+\infty}\Big|\int_{B_R(0)}\bar\rho(x)\tilde\rho^k_{\phi(s,\cdot),s}(0,x)dx\Big|ds&\leq T \int_{B_R(0)}|\bar\rho(x)|dx\sup_{s\in[0,T]}\|\tilde\rho^k_{\phi(s,\cdot),s}(0,\cdot)\|_{L^\infty_x}\\
&\leq T \int_{B_R(0)}|\bar\rho(x)|dx \|\phi\|_{L^\infty_{t,x}}.
\end{split} 
\end{equation*}
So by the dominated convergence theorem, it follows that
\begin{equation*}
\int_0^{+\infty}\int_{B_R(0)} \bar\rho(x) \tilde\rho^k_{\phi(s,\cdot),s}(0,x)dxds\underset{k\to+\infty}{\longrightarrow} \int_0^{+\infty}\int_{B_R(0)} \bar\rho(x) \tilde\rho_{\phi(s,\cdot),s}(0,x)dxds.
\end{equation*}
As $\rho^{\psi\circ\xi(k)}$ converges weakly in $L^1_{loc}([0,+\infty)\times\R^d)$ to $\rho$, we have
\begin{equation*}
\int_0^{+\infty}\int_{B_R(0)}\rho^{\psi\circ\xi(k)}(s,x)\phi(s,x)dxds\underset{k\to+\infty}{\longrightarrow} \int_0^{+\infty}\int_{B_R(0)}\rho(s,x)\phi(s,x)dxds.
\end{equation*}
We can thus pass into the limit $k\to \infty$ along the map $\psi\circ\xi:\N\to\N$ in \eqref{eqn_change_variable} and get
\begin{equation}
\int_0^{+\infty} \int_{B_R(0)} \rho(s,x)\phi(s,x)dxds=\int_0^{+\infty} \int_{B_R(0)} \bar\rho(x)\tilde\rho_{\phi(s,\cdot),s}(0,x)dxds. 
\end{equation}

\bigskip 
STEP 5 (Characterisation of the limit):
Since $\tilde\rho_{\phi(s,\cdot),s}$ does not depend on the map $\psi\circ\xi:\N\to\N$, 
by the subsubsequence lemma, it follows that the whole sequence converges, i.e. 
\begin{equation}
\begin{split} 
\lim_{k\to+\infty}\int_0^{+\infty} \int_{B_R(0)}\rho^k(s,x)\phi(s,x)dxdt&=\int_0^{+\infty} \int_{B_R(0)} \rho(s,x)\phi(s,x)dxds\\
&=\int_0^{+\infty} \int_{B_R(0)} \bar\rho(x)\tilde\rho_{\phi(s,\cdot),s}(0,x)dxds.
\end{split} 
\end{equation}
Thanks to (CS), we have 
 \begin{equation*}
\int_0^{+\infty} \int_{B_R(0)} \rho(s,x)\phi(s,x)dxds=\int_0^{+\infty}\int_{\R^d}\rho(s,x)\phi(s,x)dxds.
\end{equation*}
Therefore,
\begin{equation*}
\int_0^{+\infty}\int_{\R^d}\rho(s,x)\phi(s,x)dxds=\int_0^{+\infty} \int_{B_R(0)} \bar\rho(x)\tilde\rho_{\phi(s,\cdot),s}(0,x)dxds.
\end{equation*}}
 We thus have characterised $\rho$ in terms of the time continous representative $\tilde\rho_{\phi(s,\cdot),s}$ of the unique bounded weak solutions for \eqref{eqn_boundary_value} along $\bb$ with $s>0$, which does not depend on $\theta$. Since  $\phi\in C_c^\infty((0,+\infty)\times\R^d)$ is arbitrary, it follows that $\rho^k$ converges to $\rho$ in $\mathcal{D}'((0,+\infty)\times\R^d)$.  Since $\theta$ was an arbitrary standard mollifier, the Cauchy problem \eqref{eqPDE} is well-posed along $\mathcal{R}_{conv}$, the convolution regularisation class. }
\end{proof}

\bigskip
\section{proof of Theorem \ref{thm_depauw}}\label{section_proof_depauw} 

\subsection{Construction of the Depauw vector field $\bb_{DP}$} 
We construct the bounded, divergence-free vector field $\bb_{DP}:[0,+\infty)\times\R^2\to \R^2$ of Depauw from \cite{Depauw}, as well as two non-unique bounded weak solutions $\zeta_1$ and $\zeta_2$ of \eqref{eqPDE} along $\bb_{DP}$ with the initial datum $\bar\rho=1/2$. We follow closely the construction of a similar vector field given in \cite{DeLellis_Giri22}.

\bigskip 

Introduce the following two lattices on $\mathbb R^2$, namely $\mathcal{L}^1 := \mathbb Z^2\subset \mathbb R^2$ and $\mathcal{L}^2:=\mathbb Z^2 + (\frac{1}{2}, \frac{1}{2})\subset \mathbb R^2$. To each lattice, associate a subdivision of the plane into squares, which have vertices lying in the corresponding lattices, which we denote by $\mathcal{S}^1$ and $\mathcal{S}^2$. Then consider the rescaled lattices $\mathcal{L}^1_k:= 2^{-k} \mathbb{Z}^2$ and $\mathcal{L}^2_k := (2^{-k-1},2^{-k-1})+2^{-k} \mathbb Z^2$ and the corresponding square subdivision of $\mathbb Z^2$, respectively $\mathcal{S}^1_k$ and $\mathcal{S}^2_k$. Observe that the centres of the squares $\mathcal{S}^1_k$ are elements of $\mathcal{L}^2_k$ and viceversa.


Next, define the following $2$-dimensional autonomous vector field:
\[
\ww(x) =
\begin{cases}
(0, 4x_1)^t\text{ , if }1/2 > |x_1| > |x_2| \\
(-4x_2, 0)^t\text{ , if }1/2 > |x_2| > |x_1| \\
(0, 0)^t\text{ , otherwise.} \\
\end{cases}
\]
$\ww$ is a bounded, divergence-free vector field, whose derivative is a finite matrix-valued Radon measure given by
\begin{equation*}
\begin{split} 
D\ww(x_1,x_2) = 
&\begin{pmatrix}
0 & 0 \\
4{\rm sgn}(x_1) & 0 
\end{pmatrix} 
\Leb{d}\lfloor_{\{|x_2|<|x_1|<1/2\}} +
\begin{pmatrix}
0 & -4{\rm sgn}(x_2) \\
0 & 0 
\end{pmatrix} 
\Leb{d}\lfloor_{\{|x_1|<|x_2|<1/2\}}\\
&+\begin{pmatrix}
4x_2{\rm sgn}(x_1) & -4x_2{\rm sgn}(x_2) \\
4x_1{\rm sgn}(x_1) & -4x_1{\rm sgn}(x_2)
\end{pmatrix} 
\mathscr{H}^{d-1}\lfloor_{\{x_1=x_2,0<|x_1|,|x_2|\leq 1/2\}}\\
\end{split} 
\end{equation*}

Periodise $\ww$ by defining $\Lambda = \{(y_1, y_2) \in \mathbb{Z}^2 : y_1 + y_2 \text{ is even}\}$ and setting 
\[
\uu(x) = \sum_{y \in \Lambda} \ww(x-y)\, .
\]

Even though $\uu$ is non-smooth, it is in $BV_{loc}(\R^2;\R^2)$. 
By the theory of regular Lagrangian flows (see for instance \cite{ambrosiocrippaedi}), there exists a unique incompressible almost everywhere defined flow $\XX$ along $\uu$ can be described explicitely. 
\begin{itemize}
	\item[(R)] The map $\XX (t,0 ,\cdot)$ is Lipschitz on each square $S$ of $\mathcal{S}^2$ and $\XX (1/2,0, \cdot)$ is a clockwise rotation of $\pi/2$ radians of the ``filled'' $S$, while it is the identity on the ``empty ones''. In particular for every $j\geq 1,$ $\XX (1/2,0, \cdot)$ maps an element of $\mathcal{S}^1_j$ rigidly onto another element of $\mathcal{S}^1_j$. For $j=1$ we can be more specific. Each $S\in \mathcal{S}^2$ is formed precisely by $4$ squares of $\mathcal{S}^1_1$: in the case of ``filled'' $S$ the $4$ squares are permuted in a $4$-cycle clockwise, while in the case of ``empty'' $S$ the $4$ squares are kept fixed.  
\end{itemize}

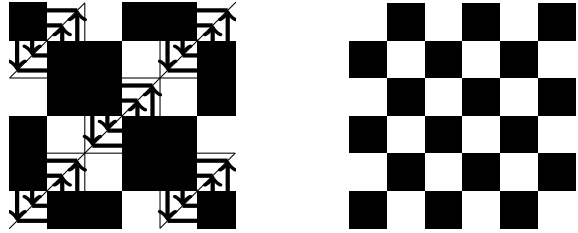
\begin{figure}[h]
	\centering
	\subfloat{
		\begin{tikzpicture}
		\clip(-1.5,-1.5) rectangle (1.5,1.5);
		\foreach \x in {-1,...,2} \foreach \y in {-1,...,2}
		{
			\pgfmathparse{mod(\x+\y,2) ? "black" : "white"}
			\edef\nu{\pgfmathresult}
			\path[fill=\nu, opacity=0.5] (\x-1,\y-1) rectangle ++ (1,1);
		}
		\draw (-1.5,0.5) -- (1.5,0.5);
		\draw (-1.5,-0.5) -- (1.5,-0.5);
		\draw (0.5, -1.5) -- (0.5, 1.5);
		\draw (-0.5, -1.5) -- (-0.5, 1.5);
		\foreach \a in {1,...,2}
		{
			\draw[ultra thick, ->] (\a/5,-\a/5) -- (\a/5,\a/5); 
			\draw[ultra thick, ->] (\a/5,\a/5) -- (-\a/5,\a/5);
			\draw[ultra thick, ->] (-\a/5,\a/5) -- (-\a/5,-\a/5);
			\draw[ultra thick, ->] (-\a/5,-\a/5) -- (\a/5,-\a/5);
		}
		\foreach \a in {1,...,2}
		{
			\draw[ultra thick, ->] (\a/5+1,-\a/5+1) -- (\a/5+1,\a/5+1); 
			\draw[ultra thick, ->] (\a/5+1,\a/5+1) -- (-\a/5+1,\a/5+1);
			\draw[ultra thick, ->] (-\a/5+1,\a/5+1) -- (-\a/5+1,-\a/5+1);
			\draw[ultra thick, ->] (-\a/5+1,-\a/5+1) -- (\a/5+1,-\a/5+1);
		}
		\foreach \a in {1,...,2}
		{
			\draw[ultra thick, ->] (\a/5+1,-\a/5-1) -- (\a/5+1,\a/5-1); 
			\draw[ultra thick, ->] (\a/5+1,\a/5-1) -- (-\a/5+1,\a/5-1);
			\draw[ultra thick, ->] (-\a/5+1,\a/5-1) -- (-\a/5+1,-\a/5-1);
			\draw[ultra thick, ->] (-\a/5+1,-\a/5-1) -- (\a/5+1,-\a/5-1);
		}
		\foreach \a in {1,...,2}
		{
			\draw[ultra thick, ->] (\a/5-1,-\a/5-1) -- (\a/5-1,\a/5-1); 
			\draw[ultra thick, ->] (\a/5-1,\a/5-1) -- (-\a/5-1,\a/5-1);
			\draw[ultra thick, ->] (-\a/5-1,\a/5-1) -- (-\a/5-1,-\a/5-1);
			\draw[ultra thick, ->] (-\a/5-1,-\a/5-1) -- (\a/5-1,-\a/5-1);
		}
		\foreach \a in {1,...,2}
		{
			\draw[ultra thick, ->] (\a/5-1,-\a/5+1) -- (\a/5-1,\a/5+1); 
			\draw[ultra thick, ->] (\a/5-1,\a/5+1) -- (-\a/5-1,\a/5+1);
			\draw[ultra thick, ->] (-\a/5-1,\a/5+1) -- (-\a/5-1,-\a/5+1);
			\draw[ultra thick, ->] (-\a/5-1,-\a/5+1) -- (\a/5-1,-\a/5+1);
		}
		\draw (-2,-2) -- (2,2);
		\draw (-2,2) -- (2,-2);
		\draw (-1.5,0.5) -- (-0.5, 1.5);
		\draw (1.5,0.5) -- (0.5,1.5);
		\draw (1.5,-0.5) -- (0.5, -1.5);
		\draw (-1.5,-0.5) -- (-0.5,-1.5);
		\end{tikzpicture}%
	}
	\qquad
	\qquad
	\subfloat
	{
		\begin{tikzpicture}[scale=0.5]
		\foreach \x in {-2,...,3} \foreach \y in {-2,...,3}
		{
			\pgfmathparse{mod(\x+\y,2) ? "white" : "black"}
			\edef\nu{\pgfmathresult}
			\path[fill=\nu, opacity=0.5] (\x-1,\y-1) rectangle ++ (1,1);
		}
		\end{tikzpicture}%
	}
	\caption{Action of the flow of $\uu$ from $t=0$ to $t=1/2$. The shaded region denotes the set $\{\zeta_1=1\}$. The figure is from \cite{DeLellis_Giri22}.}
\end{figure}

Let $\bar\zeta_1(x) = \lfloor{x_1}\rfloor + \lfloor{x_2}\rfloor \ mod \ 2$. It is a chessboard pattern based on the standard lattice $\mathbb{Z}^2 \subset \mathbb{R}^2$. Let $\zeta^1$ be the unique bounded weak solution of \eqref{eqPDE} along $\uu$ from Theorem \ref{thm_ambrosio}. Then, we have the following formula  $\XX(t,0,\cdot)_\#\bar\zeta_1(x)\Leb{d}=\zeta_1(t,x)\Leb{d}$. 
Using property (R), we have 
\begin{equation}\label{e:refining}
\zeta_1(1/2, x) = 1 - \bar\zeta_1(2x) .
\end{equation}
We define $\bb_{DP}:[0,+\infty)\times\R^2\to\R^2$ as follows. First of all, set $\bb(t,x)\equiv 0$ for $t>1$. Then, set $\bb_{DP}(t, x) = \uu(x)$ for $1/2<t<1$ and $\bb_{DP}(t, x) = \uu(2^k x)$ for $1/2^{k+1}<t<1/2^{k}$.
In particular, this yields a bounded weak solution of \eqref{eqn_boundary_value} along $\tilde\bb_{DP}$ (the extension by zero to negative times from \eqref{eqn_extension_zero}) with boundary datum $\zeta_1(1,x)=\bar\zeta_1(x)$. Moreover, using recursively the appropriately scaled version of \eqref{e:refining} we can check that 
\begin{equation*} 
\zeta_1 (1/2^{k}, x) = \bar\zeta_1 (2^{k} x) \quad\text{for $k$ even,}\qquad \zeta_1 (1/2^{k}, x) =1- \bar\zeta_1 (2^{k} x)\quad\text{ for $k$ odd. }
\end{equation*} 

Likewise, $\zeta_2(t,x)=1-\zeta_1(t,x)$ is a solution of \eqref{eqn_boundary_value} along $\tilde\bb_{DP}$ with boundary datum $\zeta_2(1,x)=\bar\zeta_2(x)=1-\bar\zeta_1(x).$ Notice also that for $i=1,2$, we have $\zeta_i(2^{-k},\cdot)$ converges weakly-star in $L^\infty(\R^d)$ to $1/2$, as $k\to +\infty$. Therefore, for $i=1,2$, we have that $\zeta_i$ are both bounded weak solutions of \eqref{eqPDE} along $\bb_{DP}$ with initial datum $\bar\rho=1/2$. 

\subsection{Properties of $\bb_{DP}$}\label{sec_prop_depauw} 
We summarise the properties of the bounded, divergence-free vector field $\bb_{DP}:[0,+\infty)\times\R^d\to\R^d$ and of the two solutions $\zeta_1$ and $\zeta_2$ we have constructed: 
\begin{enumerate} [$(a)$]
	\item $\bb_{DP}\in L_{loc}^1((0,+\infty);BV_{loc}(\R^d;\R^d))$;
	\item  $\bb_{DP}\notin   L_{loc}^1([0,+\infty);BV_{loc}(\R^d;\R^d))$;
	\item $\zeta_i$ is a bounded weak solution of \eqref{eqPDE} along $\bb_{DP}$ with initial datum $\bar\zeta_i=1/2$ for $i=1,2$;
	\item $\zeta_i\in C([0,+\infty);w^*-L^\infty(\R^d))$ for $i=1,2$;
	\item  for every $t\in(0,+\infty)$, $\zeta_i(t,x)\in\{0,1\}$ for $\Leb{d}$-a.e. $x\in\R^d$, for $i=1,2$; 
	\item for every $t\in[0,+\infty),$ we have $ \zeta_1(t,x) =1- \zeta_2(t,x)$ for $\Leb{d}$-a.e. $x\in\R^d$;
	\item\label{item_control_local_avrg} for every $k\in\N$ and for $i=1,2$, we have
	\begin{equation*}
\fint_{S} \zeta_i(2^{-(k+1)},x)dx=1/2\qquad \forall S\in \mathcal{S}_k^1.
	\end{equation*}
\end{enumerate}

\subsection{Weak approximation of initial datum} Given a sequence of indicator functions, which is oscillating infinitely fast, and satisfies a constraint of spatial homogeneity, any initial datum $\bar\rho$ can be weakly approximated by a sequence localised on these indicator functions. This is recorded in the following lemma. 
\begin{lemma}\label{lem_exist_approx_sequence} 
	Consider a sequence $(f^k)_{k\in\N}$ of functions on $\R^d$ such that for every $k\in\N$, we have $f^k(x)\in\{0,1\}$ for a.e. $x\in\R^d$, and such that and every $S\in\mathcal{S}^1_k$, we have
	\begin{equation} \label{eqn_control_local_average} 
	\fint_{S}f^{k+1}(y)dy = 1/2. 
	\end{equation}
	Consider a function $\bar\rho\in L^\infty(\R^d)$. 
	Then, there exists a sequence of functions $(\bar\rho^k)_{k\in\N}$ uniformly bounded in $L^\infty(\R^d)$ such that 
	\begin{equation}\label{eqn_support_rho^k}
	\bar\rho^kf^{k+1}=\bar\rho^k,
	\end{equation} 
	and 
	\begin{equation} \label{eqn_conv_rho^k} 
	 \bar \rho^k\rightharpoonup \bar\rho \quad\text{$w^*-L^\infty(\R^d)$}.
	\end{equation}
	{ Moreover, if $\bar\rho$ is non-negative, then the sequence $(\bar\rho^k)_{k\in\N}$ can be chosen to consist of non-negative functions.}
\end{lemma}

\begin{proof}
Define
	 \begin{equation}\label{eqn_funct_lem}
	\bar\rho^k(x):=2f^{k+1}(x)\sum_{S\in \mathcal{S}^1_k}\mathbb{1}_S(x)\fint_S\bar\rho(y)dy,
	\end{equation}
	and observe that \eqref{eqn_support_rho^k} is satisfied. 
	Note that $$\sup_{k\in\N}\|\bar\rho^k\|_{L^\infty_x}\leq 2\|\bar\rho\|_{L^\infty_x},$$
	and that finite linear combinations of functions in 
	\begin{equation*}
	\Big\{\mathbb{1}_S \, :\, S\in \mathcal{S}^1_l, \;l\in\N\Big\},
	\end{equation*}
	are dense in $L^1(\R^d)$. Therefore, to prove \eqref{eqn_conv_rho^k}, it suffices to prove that for every $l\in\N$ and every $\tilde S\in\mathcal{S}^1_l$, we have 
	\begin{equation}
	\int_{\R^d} \bar\rho^k(x)\mathbb{1}_{\tilde S}(x)dx\to \int_{\R^d}\bar\rho(x)\mathbb{1}_{\tilde S}(x)dx,
	\end{equation}
	as $k\to +\infty.$ 
	Let $l\in\N$ and $\tilde S\in \mathcal{S}^1_l$. 
	Then, for $k\geq l$, we have 
	\begin{equation}
	\begin{split}
	\int_{\R^d}\bar\rho^k(x)\mathbb{1}_{\tilde S}(x)dx&=2\int_{\tilde S}f^{k+1}(x)\sum_{S\in \mathcal{S}_k^1}\mathbb{1}_S(x)\fint_S\bar\rho(y)dydx,\\
	&=2k^d\sum_{S\in \mathcal{S}_k^1,  S\subset\tilde S}\int_{ S}f^{k+1}(x)dx\int_S\bar\rho(y)dy,\\
	\text{by \eqref{eqn_control_local_average}}&=\sum_{S\in \mathcal{S}^1_k,  S\subset \tilde S} \int_S\bar\rho(y)dy,\\
	&=\int_{\R^d}\bar\rho(y)\mathbb{1}_{\tilde S}(y)dy.
	\end{split}
	\end{equation}
	{Moreover it is clear that if $\bar\rho$ is non-negative, then so are the functions defined by \eqref{eqn_funct_lem}.}
	This proves the thesis. 
\end{proof}
\subsection{Proof of Theorem \ref{thm_depauw}} We can now conclude this section. 
\begin{proof}[Proof of Theorem \ref{thm_depauw}]
	If $\bar\rho=0$, then $\zeta_i(t,x) -1/2$ is a bounded weak solution of \eqref{eqPDE} along $\bb_{DP}$ with initial datum $\bar\rho=0$ for $i=1,2$. 
	Suppose that $\bar\rho\neq 0$. 
	Consider the following sequence of bounded, divergence-free vector fields: 
	\begin{equation*}
\bb_{DP}^k(t,x):=\left\{ 
\begin{split} 
\bb_{DP}(t,x)\qquad &\text{if} \quad t\geq 2^{-k},\\
0\qquad &\text{if}\quad t<2^{-k}.
\end{split}
\right. 
\end{equation*}

By Section \ref{sec_prop_depauw}, for $i=1,2$ the densities $\zeta_i^k$ given by: 
	\begin{equation*}
\zeta_i^k(t,x):=\left\{ 
\begin{split} 
\zeta_i(t,x)\qquad &\text{if} \quad t\geq 2^{-k},\\
\zeta_i(2^{-k},x)\qquad &\text{if}\quad t<2^{-k},
\end{split}
\right. 
\end{equation*}
are both bounded weak solutions of \eqref{eqPDE} along $\bb_{DP}^k$. 

\bigskip 
STEP 1 (Apply Lemma \ref{lem_exist_approx_sequence}):
For $i=1,2$, the sequence $(\zeta_i(2^{-k},\cdot))_{k\in\N}$ satisfies the hypothesis of Lemma \ref{lem_exist_approx_sequence} by item (\ref{item_control_local_avrg}) of Section \ref{sec_prop_depauw}. So there exists a sequence $(\bar\rho_i^k)_{k\in\N}$ uniformly bounded in $L_x^\infty$ and such that for a.e. $x\in\R^d$
\begin{equation*}
\bar\rho_i^k(x) \zeta_i(2^{-k-1},x)=\bar\rho_i^k(x),
\end{equation*}
and also 
\begin{equation*}
\bar\rho_i^k\rightharpoonup \bar\rho \qquad\text{ in $w^*-L^\infty(\R^d)$}.
\end{equation*}

\bigskip
STEP 2 (Getting weak solutions of \eqref{eqPDE} along $\bb_{DP}$): For $i=1,2$, we apply Theorem \ref{thm_uniqueness_bv}. There thus exists a unique bounded weak solution of \eqref{eqPDE} along $\bb_{DP}^k$ with initial datum $\bar\rho_i^k$ , which we shall denote by $\rho^k_i$, and satisfies $\|\rho^k_i\|_{L^\infty_{t,x}}=\|\bar\rho^k_i\|_{L^\infty_{x}}$. Therefore, by the Banach-Alaoglou Theorem, along an increasing map $\psi:\N\to\N$, we have $\rho^{\psi(k)}_i$ converge to $\rho_i$ weakly-star in $L^\infty((0,+\infty)\times\R^d)$ as $k\to+\infty$. Since $\bb_{DP}^k\to \bb_{DP}$ in $L^1_{loc}$, it follows that $\rho_i$ is a bounded weak solution of \eqref{eqPDE} along $\bb_{DP}$ with initial datum $\bar\rho$.

\bigskip 

STEP 3 (Control on $\rho_i$):
By Theorem \ref{thm_uniqueness_bv}, for $M\in\R$ there exists a unique bounded weak solution of \eqref{eqPDE} along $\bb_{DP}^k$ with initial datum $M\zeta_i(2^{-k},x)-\bar\rho^k_i(x)$ for $i=1,2$, and by linearity it is
	\begin{equation*}
\rho_i^{k,M}(t,x):=\left\{ 
\begin{split} 
M\zeta_i(t,x)-\rho^k_i(t,x)\qquad &\text{if} \quad t\geq 2^{-k},\\
M\zeta_i(2^{-k},x)-\bar\rho^k_i(x)\qquad &\text{if}\quad t<2^{-k}.
\end{split}
\right. 
\end{equation*}
Now, choose
$$M>\sup_{k\in\N,i=1,2}\|\bar\rho^k_i\|_{L_x^\infty},$$
so that $\rho^{k,-M}\leq 0\leq \rho^{k,M}$, and notice that we have that $\rho^{\psi(k),-M}_i$ and $\rho^{\psi(k),M}_i$ converge weakly-star in $L^\infty((0,+\infty)\times\R^d)$ to $M\zeta_i-\rho_i$ and $-M\zeta_i-\rho_i$ respectively as $k\to+\infty$, whence 
\begin{equation}\label{eqn_supp_rho_i}
-M\zeta_i\leq \rho_i\leq M\zeta_i \qquad \text{for $i=1,2$}.
\end{equation}

\bigskip 
STEP 4 ($\rho_1\neq \rho_2$): By Lemma \ref{lem_given_weak_sol_exist_cont_repr}, we have the time continuous representative $\tilde\rho_1$ and $\tilde\rho_2$, of $\rho_1$ and $\rho_2$ respectively. We also know that $\rho_i(0,\cdot)=\bar\rho(\cdot)$, and since
$\bar\rho\neq 0$, we have for $i=1,2$ that $$\sup_{t<\delta}\|\tilde\rho_i(t,\cdot)\|_{L^\infty_x}>0,$$ for $\delta>0$ sufficiently small. Thus, there exists
 a function $\phi\in L^1((0,+\infty)\times\R^d)$ such that for $i=1,2$, we have
\begin{equation}
\int_0^{+\infty}\int_{\R^d}|\rho_i(t,x)|\phi(t,x)dxdt=\int_0^{+\infty}\int_{\R^d}|\tilde\rho_i(t,x)|\phi(t,x)dxdt >0.
\end{equation}
Now by \eqref{eqn_supp_rho_i} and since $\zeta_i(t,x)=\mathbb{1}_{\supp\zeta_i}(t,x)$, we have 
\begin{equation}
\begin{split} 
\int_0^{+\infty}\int_{\R^d}|\rho_i(t,x)|\phi(t,x)\zeta_i(t,x)dxdt=\int_0^{+\infty}\int_{\R^d}|\rho_i(t,x)|\phi(t,x)dxdt>0.
\end{split} 
\end{equation}
Since $\zeta_1(t,x)+\zeta_2(t,x)=1$ for $\Leb{d+1}$-a.e. $(t,x)\in(0,+\infty)\times\R^d$, we have for $i,j=1,2$ with $i\neq j$
\begin{equation}
\int_0^{+\infty}\int_{\R^d}|\rho_i(t,x)|\phi(t,x)\zeta_j(t,x)dxdt=0.
\end{equation}
Since $\phi\zeta_1$ and $\phi\zeta_2$ are in $L^1((0,+\infty)\times\R^d)$, by the duality $\Big(L^1((0,+\infty)\times\R^d)\Big)^*\cong L^\infty((0,+\infty)\times\R^d)$, we have that $\rho_1$ and $\rho_2$ are distinct as functions in $L^\infty((0,+\infty)\times\R^d)$. 

\bigskip
{
STEP 5 (Non-negativity): Assume now that $\bar\rho$ is non-negative, i.e. that $\bar\rho(x)\geq 0$ for $\Leb{d}$-a.e. $x\in\R^d$ and that $\bar\rho\neq 0$. Then, by Lemma \ref{lem_exist_approx_sequence}, for every $k\in\N$, the functions $\bar\rho_1^k$ and $\bar\rho_2^k$ are non-negative. Therefore, the solutions $\rho_1^k$ and $\rho_2^k$ of \eqref{eqPDE} given in Step 2 are non-negative. Therefore, the $L^\infty$ weak-star limit points $\rho_1$ and $\rho_2$ are also non-negative. This concludes the proof. }
\end{proof}

\section{Appendix} 
\begin{proof}[Proof of Proposition \ref{prop_well_posed}]
	The unique weak solution $\rho^k$ of \eqref{eqPDE} is given by $$\rho^k(t,x)\Leb{d}=\XX^k(t,0,x)_\#\bar\rho\Leb{d}.$$
	By $(ii)$, for any $T>0$ we have $$\sup_{t\in[0,T]}\|\rho^k(t,\cdot)\|_{L^\infty_{x}}\leq C\|\bar\rho\|_{L^\infty_x}.$$
	Let $\xi:\N\to\N$ be an increasing map. Then, by the Banach-Alaoglou, there exists an increasing map $\psi:\N\to\N$, and $\rho\in L^\infty_{loc}([0,+\infty);L^\infty(\R^d))$ such that $\rho^{\psi\circ\xi(k)}$ converges to $\rho$ in weakly-star in $L^\infty((0,T)\times\R^d)$ for every $T>0$. 
	As $\bb^k$ converges strongly to $\bb$ in $L^1_{loc}$, $\rho$ is therefore a bounded weak solution of \eqref{eqPDE} along $\bb$. By Theorem \ref{thm_uniqueness_bv}, there exists a unique such bounded weak solution. The subsubsequence lemma then implies that the whole sequence $\rho^k$ converges to $\rho$ weakly-star in $L^\infty((0,T)\times  \R^d))$ for every $T>0$. This implies that $\rho^k$ converges to $\rho$ in $\mathcal{D}'((0,+\infty)\times\R^d)$. The thesis follows.
\end{proof}

\bibliographystyle{alpha}
\bibliography{bibliografia}

\end{document}